\documentclass[10pt,reqno]{amsart}
\usepackage[hypertex]{hyperref}

\usepackage{amsmath,amsthm}
\usepackage{amssymb}
\usepackage{euscript}

\numberwithin{equation}{subsection}

\newcommand{\sqsp}{\renewcommand{\baselinestretch}{1.15}\tiny\normalsize}

\newcommand{\relphantom[1]}{\mathrel{\phantom{#1}}}

\newtheorem{theorem}[subsection]{Theorem}
\newtheorem{lemma}[subsection]{Lemma}
\newtheorem{proposition}[subsection]{Proposition}
\newtheorem{corollary}[subsection]{Corollary}

\theoremstyle{definition}
\newtheorem{definition}[subsection]{Definition}
\newtheorem{example}[subsection]{Example}

\newcommand{\bk}{\mathbf{k}}
\newcommand{\bS}{\mathbf{S}}

\newcommand{\xbar}{\overline{x}}
\newcommand{\ybar}{\overline{y}}
\newcommand{\zbar}{\overline{z}}

\newcommand{\muop}{\mu^{op}}
\newcommand{\mualpha}{\mu_\alpha}
\newcommand{\mubeta}{\mu_\beta}
\newcommand{\mulambda}{\mu_\lambda}
\newcommand{\mun}{\mu^{(n)}}

\newcommand{\cyclic}{(Id + \pi + \pi^2)}
\newcommand{\sign}{(-1)^{\varepsilon(\sigma)}}

\newcommand{\oct}{\mathbf{O}}
\newcommand{\octalpha}{\mathbf{O}_\alpha}

\newcommand{\sedenion}{{
\begin{tiny}
\begin{center}
\begin{tabular}{|c|c|c|c|c|c|c|c|c|c|c|c|c|c|c|c|c|}\hline
 & $1$ & $e_1$ & $e_2$ & $e_3$ & $e_4$ & $e_5$ & $e_6$ & $e_7$ & $e_8$ & $e_9$ & $e_{10}$ & $e_{11}$ & $e_{12}$ & $e_{13}$ & $e_{14}$ & $e_{15}$ \\ \hline
$1$ & $1$ & $e_1$ & $e_2$ & $e_3$ & $e_4$ & $e_5$ & $e_6$ & $e_7$ & $e_8$ & $e_9$ & $e_{10}$ & $e_{11}$ & $e_{12}$ & $e_{13}$ & $e_{14}$ & $e_{15}$ \\ \hline
$e_1$ & $e_1$ & $-1$ & $e_3$ & $-e_2$ & $e_5$ & $-e_4$ & $-e_7$ & $e_6$ & $e_9$ & $-e_8$ & $-e_{11}$ & $e_{10}$ & $-e_{13}$ & $e_{12}$ & $e_{15}$ & $-e_{14}$ \\ \hline
$e_2$ & $e_2$ & $-e_3$ & $-1$ & $e_1$ & $e_6$ & $e_7$ & $-e_4$ & $-e_5$ & $e_{10}$ & $e_{11}$ & $-e_8$ & $-e_9$ & $-e_{14}$ & $-e_{15}$ & $e_{12}$ & $e_{13}$ \\ \hline
$e_3$ & $e_3$ & $e_2$ & $-e_1$ & $-1$ & $e_7$ & $-e_6$ & $e_5$ & $-e_4$ & $e_{11}$ & $-e_{10}$ & $e_9$ & $-e_8$ & $-e_{15}$ & $e_{14}$ & $-e_{13}$ & $e_{12}$ \\ \hline
$e_4$ & $e_4$ & $-e_5$ & $-e_6$ & $-e_7$ & $-1$ & $e_1$ & $e_2$ & $e_3$ & $e_{12}$ & $e_{13}$ & $e_{14}$ & $e_{15}$ & $-e_8$ & $-e_9$ & $-e_{10}$ & $-e_{11}$ \\ \hline
$e_5$ & $e_5$ & $e_4$ & $-e_7$ & $e_6$ & $-e_1$ & $-1$ & $-e_3$ & $e_2$ & $e_{13}$ & $-e_{12}$ & $e_{15}$ & $-e_{14}$ & $e_9$ & $-e_8$ & $e_{11}$ & $-e_{10}$ \\ \hline
$e_6$ & $e_6$ & $e_7$ & $e_4$ & $-e_5$ & $-e_2$ & $e_3$ & $-1$ & $-e_1$ & $e_{14}$ & $-e_{15}$ & $-e_{12}$ & $e_{13}$ & $e_{10}$ & $-e_{11}$ & $-e_8$ & $e_9$ \\ \hline
$e_7$ & $e_7$ & $-e_6$ & $e_5$ & $e_4$ & $-e_3$ & $-e_2$ & $e_1$ & $-1$ & $e_{15}$ & $e_{14}$ & $-e_{13}$ & $-e_{12}$ & $e_{11}$ & $e_{10}$ & $-e_9$ & $-e_8$ \\ \hline
$e_8$ & $e_8$ & $-e_9$ & $-e_{10}$ & $-e_{11}$ & $-e_{12}$ & $-e_{13}$ & $-e_{14}$ & $-e_{15}$ & $-1$ & $e_1$ & $e_2$ & $e_3$ & $e_4$ & $e_5$ & $e_6$ & $e_7$ \\ \hline
$e_9$ & $e_9$ & $e_8$ & $-e_{11}$ & $e_{10}$ & $-e_{13}$ & $e_{12}$ & $e_{15}$ & $-e_{14}$ & $-e_1$ & $-1$ & $-e_3$ & $e_2$ & $-e_5$ & $e_4$ & $e_7$ & $-e_6$ \\ \hline
$e_{10}$ & $e_{10}$ & $e_{11}$ & $e_8$ & $-e_9$ & $-e_{14}$ & $-e_{15}$ & $e_{12}$ & $e_{13}$ & $-e_2$ & $e_3$ & $-1$ & $-e_1$ & $-e_6$ & $-e_7$ & $e_4$ & $e_5$ \\ \hline
$e_{11}$ & $e_{11}$ & $-e_{10}$ & $e_9$ & $e_8$ & $-e_{15}$ & $e_{14}$ & $-e_{13}$ & $e_{12}$ & $-e_3$ & $-e_2$ & $e_1$ & $-1$ & $-e_7$ & $e_6$ & $-e_5$ & $e_4$  \\ \hline
$e_{12}$ & $e_{12}$ & $e_{13}$ & $e_{14}$ & $e_{15}$ & $e_8$ & $-e_9$ & $-e_{10}$ & $-e_{11}$ & $-e_4$ & $e_5$ & $e_6$ & $e_7$ & $-1$ & $-e_1$ & $-e_2$ & $-e_3$  \\ \hline
$e_{13}$ & $e_{13}$ & $-e_{12}$ & $e_{15}$ & $-e_{14}$ & $e_9$ & $e_8$ & $e_{11}$ & $-e_{10}$ & $-e_5$ & $-e_4$ & $e_7$ & $-e_6$ & $e_1$ & $-1$ & $e_3$ & $-e_2$  \\ \hline
$e_{14}$ & $e_{14}$ & $-e_{15}$ & $-e_{12}$ & $e_{13}$ & $e_{10}$ & $-e_{11}$ & $e_8$ & $e_9$ & $-e_6$ & $-e_7$ & $-e_4$ & $e_5$ & $e_2$ & $-e_3$ & $-1$ & $e_1$ \\ \hline
$e_{15}$ & $e_{15}$ & $e_{14}$ & $-e_{13}$ & $-e_{12}$ & $e_{11}$ & $e_{10}$ & $-e_9$ & $e_8$ & $-e_7$ & $e_6$ & $-e_5$ & $-e_4$ & $e_3$ & $e_2$ & $-e_1$ & $-1$  \\ \hline
\end{tabular}
\end{center}
\end{tiny}
}}

\newcommand{\matrixy}{{\begin{pmatrix}
1 & e_2 & e_3\\
-e_2 & 0 & 0\\
-e_3 & 0 & 0\end{pmatrix}}}

\begin{document}

\title{Hom-power associative algebras}
\author{Donald Yau}

\begin{abstract}
A generalization of power associative algebra, called Hom-power associative algebra, is studied.  The main result says that a multiplicative Hom-algebra is Hom-power associative if and only if it satisfies two identities of degrees three and four.  It generalizes Albert's result that power associativity is equivalent to third and fourth power associativity.  In particular, multiplicative right Hom-alternative algebras and non-commutative Hom-Jordan algebras are Hom-power associative.
\end{abstract}



\address{Department of Mathematics\\
    The Ohio State University at Newark\\
    1179 University Drive\\
    Newark, OH 43055, USA}
\email{dyau@math.ohio-state.edu}

\date{\today}
\maketitle

\sqsp

%%%%%%%%%%%%%%%%%%%%%%
\section{Introduction}
%%%%%%%%%%%%%%%%%%%%%%

% Albert's power associative algebras, characterization in terms of right powers, subclasses, third and fourth imply pa; sedenions

Power associative algebras were first systematically studied by Albert \cite{albert1,albert2}.  These non-associative algebras play an important role in the studies of associative bilinear forms, loops, the Cayley-Dickson algebras, and Banach algebras, among many other subjects \cite{arenas,cawagas,velasco}.  A power associative algebra can be defined conceptually as an algebra in which every element generates an associative sub-algebra.  Equivalently, power associative algebras can be algebraically characterized as follows.  For an element $x$ in an algebra $A$, define the right powers
\begin{equation}
\label{rpower}
x^1 = x \quad\text{and}\quad x^{n+1} = x^nx
\end{equation}
for $n \geq 1$.  Then an algebra $A$ is power associative if and only if
\begin{equation}
\label{pass}
x^n = x^{n-i}x^i
\end{equation}
for all $x \in A$, $n \geq 2$, and $i \in \{1,\ldots,n-1\}$.  Therefore, one can also think of \eqref{pass} as the defining identities for power associative algebras.  If $A$ satisfies \eqref{pass} for a particular integer $n$, then $A$ is called $n$th power associative.  A theorem of Albert \cite{albert1} says that (in characteristic $0$, which is assumed throughout this paper) an algebra is power associative if and only if it is third and fourth power associative, which in turn is equivalent to
\begin{equation}
\label{34}
x^2x = xx^2 \quad\text{and}\quad x^4 = x^2x^2
\end{equation}
for all elements $x$.

Associative algebras are clearly power associative, but the class of power associative algebras includes some important classes of non-associative algebras.  For example, alternative algebras and Jordan algebras are all power associative because they satisfy \eqref{34}.  In fact, right-alternative and non-commutative Jordan algebras are power associative algebras (Corollary \ref{cor4:main}).  Therefore, power associative algebras are common generalizations of right-alternative algebras and non-commutative Jordan algebras.  Moreover, there exist power associative algebras that are neither alternative nor non-commutative Jordan.  One such power associative algebra is the sedenions, which is the sixteen-dimensional algebra constructed from the octonions via the Cayley-Dickson process with the parameter $-1$ (Example \ref{ex:sedenion}).

% purpose: Hom-version including Hom-alternative and Hom-Jordan; third and fourth imply Hom-associative; constructible from pa algebras
% history of Hom-algebras

The purpose of this paper is to study a generalization of power associative algebras, called Hom-power associative algebras, in which the defining identities \eqref{pass} are relaxed by a certain linear self-map, called the twisting map.  Hom-type algebras appeared in \cite{hls} in the form of Hom-Lie algebras, in which the Jacobi identity is relaxed by a twisting map, to describe the structures on some deformations of the Witt and the Virasoro algebras.  Hom-Lie and Hom-associative algebras are further studied in \cite{ms,ms2,yau,yau2} and the references therein.  Hom-type generalizations of other algebraic structures, including quantum groups, the Yang-Baxter equations, Novikov algebras, and Lie and infinitesimal bialgebras, can be found in \cite{yau3} - \cite{yau15} and the references therein.

One reason for considering Hom-power associative algebras is that the class of Hom-power associative algebras has a certain closure property with respect to twisting by self-morphisms that the class of power associative algebras does not have.  Moreover, Hom-power associative algebras retain some important properties of power associative algebras.  For example, the main result of this paper says that  when the twisting map is multiplicative, Hom-power associativity, which involves infinitely many defining identities, is equivalent to the Hom-versions of the two identities \eqref{34}.  As a consequence of this result, multiplicative right Hom-alternative algebras and non-commutative Hom-Jordan algebras are Hom-power associative.

% organization

The rest of this paper is organized as follows.

In section \ref{sec:basic} we define Hom-type analogues of the right powers and power associativity, called Hom-powers and Hom-power associativity.  The closure property mentioned above is discussed in Theorem \ref{thm:twist}.  It says that a Hom-power associative algebra gives rise to another one when its structure maps (the multiplication and the twisting map) are both twisted by any self-morphism.  If this twisting construction is applied to a power associative algebra, regarded as a Hom-power associative algebra with identity twisting map, the result is a multiplicative Hom-power associative algebra (Corollary \ref{cor2:twist}) that is usually not power associative.  We illustrate this result in Example \ref{ex:sedenion} with the sedenions.  Given any Hom-algebra $A$, one can associate a family of Hom-algebras $A(\lambda)$ with the same underlying vector space, where $\lambda$ is any parameter in the ground field, in which the multiplication is an interpolation between those in $A$ and its opposite algebra.  We show that when $A$ is Hom-power associative, the Hom-powers in $A$ and $A(\lambda)$ coincide (Theorem \ref{thm:alambda}), and $A(\lambda)$ is also Hom-power associative (Corollary \ref{cor1:alambda}).

In preparation for the main result in section \ref{sec:maintheorem}, in sections \ref{sec:third} and \ref{sec:fourth} we study third and fourth Hom-power associativity, respectively.  In Theorem \ref{thm:third} several characterizations of third Hom-power associativity are given.  One consequence of this result is that up to $(n-1)$st Hom-power associativity implies certain commutative identities involving the Hom-powers $x^{n-i}$ and $x^i$ (Proposition \ref{prop:n-1}).  It is also observed that the intersection of the classes of third Hom-power associative algebras and of Hom-Lie admissible algebras is precisely the class of $A_3$-Hom-associative algebras (Theorem \ref{thm:a3}).  This last result is not used in later sections.

In section \ref{sec:fourth} it is observed that multiplicative right Hom-alternative and non-commutative Hom-Jordan algebras are third and fourth Hom-power associative (Proposition \ref{prop:homjordan}).  A conceptual difference between power associative algebras and Hom-power associative algebras is discussed (Examples \ref{ex:oct} and \ref{ex:jordan}).  Theorem \ref{thm:fourth} provides the linearized form of the identity $x^4 = \alpha(x^2)\alpha(x^2)$, which is part of fourth Hom-power associativity.  This result is crucial to the proof of the main result.

The main result of this paper (Corollary \ref{cor1:main}) is proved in section \ref{sec:maintheorem}.  It says that for a multiplicative Hom-algebra, Hom-power associativity is equivalent to up to fourth Hom-power associativity, which in turn is equivalent to two identities involving the first few Hom-powers.  In particular, for a multiplicative Hom-flexible algebra, Hom-power associativity is equivalent to one identity (Corollary \ref{cor1.5:main}).  Another consequence of the main result is that multiplicative right Hom-alternative and non-commutative Hom-Jordan algebras are Hom-power associative (Corollary \ref{cor3:main}).

%%%%%%%%%%%%%%%%%%%%%%%%%%%%%%%%%%%%%%%%%%%%
\section{Basic properties of Hom-power associative algebras}
\label{sec:basic}
%%%%%%%%%%%%%%%%%%%%%%%%%%%%%%%%%%%%%%%%%%%%

In this section we introduce Hom-power associative algebras and provide some basic construction results and examples.

%%%%%%%%%%%%%%%%%%%%%%%%
\subsection{Notations}

Throughout the rest of this paper, we work over a fixed field $\bk$ of characteristic $0$.  If $V$ is a $\bk$-module and $\mu \colon V^{\otimes 2} \to V$ is a bilinear map, then $\muop \colon V^{\otimes 2} \to V$ denotes the opposite map, i.e., $\muop = \mu\tau$, where $\tau \colon V^{\otimes 2} \to V^{\otimes 2}$ interchanges the two variables.  For $x$ and $y$ in $V$, we sometimes write $\mu(x,y)$ as $xy$.  For a linear self-map $\alpha \colon V \to V$, denote by $\alpha^n$ the $n$-fold composition of $n$ copies of $\alpha$, with $\alpha^0 \equiv Id$.

The symmetric group on $n$ letters is denoted by $S_n$.  A multi-linear map $f \colon V^{\otimes n} \to V$ is said to be \textbf{totally anti-symmetric} if
\[
f = \sign f \circ \sigma
\]
for all $\sigma \in S_n$, where $\varepsilon(\sigma)$ is the signature of $\sigma$.

Let us begin with the basic definitions regarding Hom-algebras.

%%%%%%%%%%%%%%%%%%
\begin{definition}
\label{def:homalg}
\begin{enumerate}
\item
By a \textbf{Hom-algebra} we mean a triple $(A,\mu,\alpha)$ in which $A$ is a $\bk$-module, $\mu \colon A^{\otimes 2} \to A$ is a bilinear map (the multiplication), and $\alpha \colon A \to A$ is a linear map (the twisting map).  Such a Hom-algebra is often abbreviated to $A$.
\item
A Hom-algebra $(A,\mu,\alpha)$ is \textbf{multiplicative} if $\alpha$ is multiplicative with respect to $\mu$, i.e., $\alpha\mu = \mu \alpha^{\otimes 2}$.
\item
A Hom-algebra $(A,\mu,\alpha)$ is \textbf{commutative} if its multiplication is commutative, i.e., $\mu = \muop$.
\item
Let $A$ and $B$ be Hom-algebras.  A \textbf{morphism} $f \colon A \to B$ of Hom-algebras is a linear map $f \colon A \to B$ such that $f\mu_A = \mu_B f^{\otimes 2}$ and $f\alpha_A = \alpha_B f$.
\end{enumerate}
\end{definition}
%%%%%%%%%%%%%%%%%%

If $(A,\mu)$ is any algebra, then we also regard it as a Hom-algebra $(A,\mu,Id)$ with identity twisting map.

Next we define the Hom-versions of Albert's right powers and power associativity \cite{albert1,albert2}.

%%%%%%%%%%%%%%%%%%%%
\begin{definition}
\label{def:hompower}
Let $(A,\mu,\alpha)$ be a Hom-algebra and $x \in A$.
\begin{enumerate}
\item
Define the \textbf{$n$th Hom-power} $x^n \in A$ inductively by
\begin{equation}
\label{hompower}
x^1 = x, \qquad
x^n = x^{n-1}\alpha^{n-2}(x)
\end{equation}
for $n \geq 2$.
\item
We say that $A$ is \textbf{$n$th Hom-power associative} if
\begin{equation}
\label{nhpa}
x^n = \alpha^{n-i-1}(x^i) \alpha^{i-1}(x^{n-i})
\end{equation}
for all $x \in A$ and $i \in \{1,\ldots, n-1\}$.
\item
We say that $A$ is \textbf{up to $n$th Hom-power associative} if $A$ is $k$th Hom-power associative for all $k \in \{2,\ldots,n\}$.
\item
We say that $A$ is \textbf{Hom-power associative} if $A$ is $n$th Hom-power associative for all $n \geq 2$.
\end{enumerate}
\end{definition}
%%%%%%%%%%%%%%%%%%%%

Some remarks are in order.  First, if $\alpha = Id$ in Definition \ref{def:hompower}, then $x^n = x^{n-1}x$, which coincides with Albert's notion of \textbf{right power} \cite{albert1,albert2}.  Moreover, in this case, $n$th Hom-power associativity reduces to $n$th \textbf{power associativity}, i.e., $x^n = x^ix^{n-i}$ for $1\leq i \leq n-1$.  Therefore, a \textbf{power associative algebra} is equivalent to a Hom-power associative algebra whose twisting map is the identity map.

Second, the first few Hom-powers are
\[
\begin{split}
x^1 &= x,\quad x^2 = xx, \quad x^3 = x^2\alpha(x) = (xx)\alpha(x),\\
x^4 &= x^3\alpha^2(x) = (x^2\alpha(x))\alpha^2(x) = ((xx)\alpha(x))\alpha^2(x),\\
x^5 &= x^4\alpha^3(x) = \left((x^2\alpha(x))\alpha^2(x)\right)\alpha^3(x) = \left(((xx)\alpha(x))\alpha^2(x)\right)\alpha^3(x).
\end{split}
\]
Note that the $i=n-1$ case of $n$th Hom-power associativity \eqref{nhpa} is automatically true because it says $x^n = x^{n-1}\alpha^{n-2}(x)$, which holds by definition.  It is included in the definition for symmetry.  Moreover, every Hom-algebra is automatically second Hom-power associative, since with $n=2$ the condition \eqref{nhpa} says $x^2 = xx$, which holds by definition.  Therefore, the first nontrivial case is third Hom-power associativity, which will be studied in the next section.

Third, by definition $n$th Hom-power associativity \eqref{nhpa} requires $n-1$ identities.  Under certain commutative conditions, the number of identities that need to be satisfied for $n$th Hom-power associativity can be reduced by half.  In fact, the two factors $\alpha^{n-i-1}(x^i)$ and $\alpha^{i-1}(x^{n-i})$ on the right-hand side of \eqref{nhpa} are related by an interchange of $i$ and $n-i$.  Therefore, if these two elements commute for all $i \in \{1,\ldots,n-1\}$, then $n$th Hom-power associativity reduces to $n/2$ or $(n-1)/2$ identities, according to whether $n$ is even or odd.  These commutative conditions will appear in Proposition \ref{prop:n-1}, which will be used in the proof of Theorem \ref{mainthm}.

Our first result says that the category of ($n$th) Hom-power associative algebras is closed under twisting by self-morphisms.

%%%%%%%%%%%%%%%%%%%
\begin{theorem}
\label{thm:twist}
Let $(A,\mu,\alpha)$ be a Hom-algebra and $\beta \colon A \to A$ be a morphism.  Define the Hom-algebra
\begin{equation}
\label{abeta}
A_\beta = (A,\mubeta = \beta\mu,\beta\alpha).
\end{equation}
Then the following statements hold.
\begin{enumerate}
\item
If $A$ is multiplicative, then so is $A_\beta$.
\item
We have
\begin{equation}
\label{abetapower}
(x^n)' = \beta^{n-1}(x^n)
\end{equation}
for all $x \in A$ and $n \geq 1$, where $x^n$ and $(x^n)'$ are the $n$th Hom-powers in $A$ and $A_\beta$, respectively.
\item
If $A$ is $n$th Hom-power associative, then so is $A_\beta$.  Therefore, if $A$ is Hom-power associative, then so is $A_\beta$.
\end{enumerate}
\end{theorem}
%%%%%%%%%%%%%%%%%%%

\begin{proof}
The first assertion regarding multiplicativity is obvious.  The second assertion is proved by induction, with the $n=1$ case being trivially true.  For the induction step, suppose that \eqref{abetapower} is true.  Using the fact that $\beta$ is a morphism, we have
\[
\begin{split}
(x^{n+1})' &= \mubeta((x^n)', (\beta\alpha)^{n-1}(x))\\
&= \mubeta(\beta^{n-1}(x^n), \beta^{n-1}\alpha^{n-1}(x))\\
&= \beta^n\mu(x^n,\alpha^{n-1}(x))\\
&= \beta^n(x^{n+1}).
\end{split}
\]
This finishes the induction and proves the second assertion.

For the last assertion, assume that $A$ is $n$th Hom-power associative and $i \in \{1,\ldots,n-1\}$.  We must show \eqref{nhpa} for $A_\beta$.  Using \eqref{abetapower} and the fact that $\beta$ is a morphism, we have
\[
\begin{split}
(x^n)' &= \beta^{n-1}(x^n)\\
&= \beta^{n-1}\mu(\alpha^{n-i-1}(x^i),\alpha^{i-1}(x^{n-i}))\\
&= \mubeta(\beta^{n-2}\alpha^{n-i-1}(x^i),\beta^{n-2}\alpha^{i-1}(x^{n-i}))\\
&= \mubeta\left((\beta\alpha)^{n-i-1}(\beta^{i-1}(x^i)), (\beta\alpha)^{i-1}(\beta^{n-i-1}(x^{n-i}))\right)\\
&= \mubeta\left((\beta\alpha)^{n-i-1}(x^i)', (\beta\alpha)^{i-1}(x^{n-i})'\right).
\end{split}
\]
This shows that $A_\beta$ is $n$th Hom-power associative.
\end{proof}

Now we discuss some consequences of Theorem \ref{thm:twist}.  The result below says that every multiplicative Hom-power associative algebra induces a derived sequence of multiplicative Hom-power associative algebras with twisted structure maps.

%%%%%%%%%%%%%%%%%%%
\begin{corollary}
\label{cor1:twist}
Let $(A,\mu,\alpha)$ be a multiplicative Hom-power associative algebra.  Then
\[
A^n = (A,\mun = \alpha^{2^n-1}\mu,\alpha^{2^n})
\]
is a multiplicative Hom-power associative algebra for each $n \geq 1$.
\end{corollary}
%%%%%%%%%%%%%%%%%%%

\begin{proof}
That $A^1$ is a multiplicative Hom-power associative algebra follows from the $\beta = \alpha$ case of Theorem \ref{thm:twist}.  Then we use induction and the fact $A^{n+1} = (A^n)^1$ to finish the proof.
\end{proof}

The next result says that multiplicative Hom-power associative algebras can be constructed from power associative algebras and their self-morphisms.  A result of this form first appeared in \cite{yau2}.

%%%%%%%%%%%%%%%%%%%
\begin{corollary}
\label{cor2:twist}
Let $A$ be a power associative algebra and $\beta \colon A \to A$ be an algebra morphism.  Then
\[
A_\beta = (A,\mubeta = \beta\mu,\beta)
\]
is a multiplicative Hom-power associative algebra.
\end{corollary}
%%%%%%%%%%%%%%%%%%%

\begin{proof}
This is the $\alpha = Id$ case of Theorem \ref{thm:twist}.
\end{proof}

%%%%%%%%%%%%%%%%%%%
\begin{example}[\textbf{Hom-Sedenions}]
\label{ex:sedenion}
In this example, let $\bk$ be the field of real numbers.  Let $\bS$ be the $16$-dimensional algebra constructed from $\bk$ (whose involution is taken to be the identity map) by applying the Cayley-Dickson process four times with the parameter $-1$ in each step  \cite{schafer}.  Elements in the non-commutative algebra $\bS$ are called sedenions.  Here we show how Corollary \ref{cor2:twist} can be applied to $\bS$ to create Hom-power associative algebras.

Let us first recall some basic properties of $\bS$ \cite{carmody,imaeda}.  Unlike the $8$-dimensional Cayley octonions, $\bS$ is not an alternative algebra, but it is power associative.  Let us provide proofs of these facts.  From the Cayley-Dickson process itself, the algebra $\bS$ has a standard basis $\{e_i \colon 0 \leq i \leq 15\}$, with $e_0 = 1$ a two-sided multiplicative identity.  Moreover, for $i \not= j$ in $\{1,\ldots,15\}$, we have
\[
e_i^2 = -1 \quad\text{and}\quad e_ie_j = -e_je_i.
\]
The multiplication table for $\bS$ using this basis is given below.

\sedenion

To see that $\bS$ is not an alternative algebra, recall that an algebra is said to be \emph{left alternative} if it satisfies
\[
(xx)y = x(xy)
\]
for all elements $x$ and $y$.  If $x = e_1 + e_8$ and $y = e_3 + e_9$ in $\bS$, then we have
\[
(xx)y = -2(e_3 + e_9), \quad x(xy) = - e_3 + e_5 -2e_9 - e_{10} - e_{12}.
\]
Therefore, $\bS$ is not left alternative.

To see that $\bS$ is power associative, let $x = \sum_{i=0}^{15} \lambda_ie_i$ be an arbitrary sedenion with each $\lambda_i \in \bk$.  Then we have
\[
x^2 = a + bx \quad \text{with} \quad a = -\sum_{i=0}^{15} \lambda_i^2,\quad b = 2\lambda_0.
\]
It follows that
\[
x^2x = ab + (a + b^2)x = xx^2
\]
and
\[
(x^2x)x = (a^2 + ab^2) + (2ab + b^3)x = x^2x^2.
\]
Therefore, by \cite{albert1} (Theorem 2) (or, equivalently, Corollary \ref{cor2:main} below), $\bS$ is power associative.

In order to apply Corollary \ref{cor2:twist} to the power associative algebra $\bS$, consider, for example, the algebra (iso)morphism $\alpha  \colon \bS \to \bS$ with
\[
\alpha(1) = 1,\quad
\alpha(e_1) = e_5,\quad
\alpha(e_2) = e_7,\quad
\alpha(e_4) = e_6,\quad
\alpha(e_8) = e_{15}.
\]
The action of $\alpha$ on the other basis elements are then uniquely determined by the multiplication table above:
\[
\begin{split}
\alpha(e_3) &= e_2,\quad
\alpha(e_5) = -e_3,\quad
\alpha(e_6) = e_1,\quad
\alpha(e_7) = -e_4,\quad
\alpha(e_9) = -e_{10},\quad
\alpha(e_{10}) = -e_8,\\
\alpha(e_{11}) &= e_{13},\quad
\alpha(e_{12}) = e_9,\quad
\alpha(e_{13}) = -e_{12},\quad
\alpha(e_{14}) = -e_{14},\quad
\alpha(e_{15}) = e_{11}.
\end{split}
\]
By Corollary \ref{cor2:twist} there is a multiplicative Hom-power associative algebra
\[
\bS_\alpha = (\bS,\mualpha = \alpha\mu,\alpha),
\]
where $\mu$ is the multiplication in $\bS$.  Observe that $(\bS,\mualpha)$ is not a power associative algebra.  Indeed, with $x = 1 + e_1$, we have
\[
\mualpha(\mualpha(x,x),x) = 2(-e_3 + e_6),\quad
\mualpha(x,\mualpha(x,x)) = 2(-e_3 - e_6).
\]
Therefore, $(\bS,\mualpha)$ is not third power associative, and hence not power associative.

Other algebra isomorphisms on $\bS$ can be constructed following the same pattern as in the previous paragraph.  In more details, let us extend the terminology in \cite{baez} for the Cayley octonions and define a \textbf{basic quadruple} in $\bS$ to be a quadruple $(e_{i_1},e_{i_2},e_{i_3},e_{i_4})$ of distinct basis elements, none of which is $1$, such that
\[
e_{i_3} \not= \pm e_{i_1}e_{i_2} \quad\text{and}\quad
e_{i_4} \not\in \pm\{e_{i_1}e_{i_2}, e_{i_1}e_{i_3}, e_{i_2}e_{i_3}, (e_{i_1}e_{i_2})e_{i_3}\}.
\]
The assumptions on a basic quadruple guarantee that it generates $\bS$ as an algebra and that the other eleven non-identity basis elements can be written uniquely in terms of the given four.  Given another basic quadruple $(e_{j_1},e_{j_2},e_{j_3},e_{j_4})$, there is a unique algebra isomorphism $\beta \colon \bS \to \bS$ determined by
\[
\beta(e_{i_l}) = e_{j_l}
\]
for $1 \leq l \leq 4$.  Corollary \ref{cor2:twist} can then be applied to $\beta$, which yields a multiplicative Hom-power associative algebra $\bS_\beta = (\bS,\mu_\beta = \beta\mu,\beta)$.  In this context, the map $\alpha$ in the previous paragraph is the unique map that sends the basic quadruple $(e_{1},e_{2},e_{4},e_{8})$ to the basic quadruple $(e_{5},e_{7},e_{6},e_{15})$.
\qed
\end{example}
%%%%%%%%%%%%%%%%%%%

In Corollary \ref{cor1:twist} it was shown that every multiplicative Hom-power associative algebra gives rise to a sequence of multiplicative Hom-power associative algebras whose structure maps are twisted by its own twisting map.  In the rest of this section, we provide another construction of Hom-power associative algebras from a given one.

%%%%%%%%%%%%%%%%%%%
\begin{definition}
\label{def:alambda}
Let $(A,\mu,\alpha)$ be a Hom-algebra.  Define the Hom-algebra
\[
A(\lambda) = (A,\mulambda = \lambda\mu + (1-\lambda)\muop,\alpha),
\]
in which $\lambda \in \bk$ is an arbitrary scalar.
\end{definition}
%%%%%%%%%%%%%%%%%%%

The multiplication $\mulambda$ is an interpolation between $\mu$ and its opposite multiplication $\muop$.  In particular, we have $A(1) = A$ and $A(0) = (A,\muop,\alpha)$.  Moreover, with $\lambda = 1/2$ we have
\[
A(1/2) = A^+ = (A,\ast = (\mu + \muop)/2,\alpha),
\]
which is commutative.  Also note that $A$ is multiplicative if and only if $A(\lambda)$ is multiplicative for all $\lambda$.

The following result is the Hom-type analogue of \cite{albert2} (Theorem 1, p.581).

%%%%%%%%%%%%%%%%%%%
\begin{theorem}
\label{thm:alambda}
Let $(A,\mu,\alpha)$ be a Hom-algebra, and let $k \geq 2$ be an integer.  Then the following statements hold.
\begin{enumerate}
\item
$A$ is up to $k$th Hom-power associative if and only if $A(\lambda)$ is up to $k$th Hom-power associative for all $\lambda \in \bk$.
\item
If $A$ is up to $k$th Hom-power associative, then the $l$th Hom-powers in $A$ and $A(\lambda)$ coincide for $l \in \{1,\ldots,k\}$.
\end{enumerate}
\end{theorem}
%%%%%%%%%%%%%%%%%%%

\begin{proof}
The ``if" part of the first assertion is obvious because $A(1) = A$.  To prove the other direction, suppose $A$ is up to $k$th Hom-power associative.  Pick any $\lambda \in \bk$.  Let us write $x^l$ and $(x^l)'$ for the $l$th Hom-powers in $A$ and $A(\lambda)$, respectively.  First we prove the second assertion, i.e.,
\begin{equation}
\label{alambdapower}
x^l = (x^l)'
\end{equation}
for $l \in \{1,\ldots,k\}$.  This is clearly true when $l=1$.  For the induction step, suppose $1 \leq l < k$.  Then
\[
\begin{split}
(x^{l+1})' &= \mulambda((x^l)',\alpha^{l-1}(x))\\
&= \mulambda(x^l,\alpha^{l-1}(x)) \quad\text{(by induction hypothesis)}\\
&= \lambda\mu(x^l,\alpha^{l-1}(x)) + (1-\lambda)\mu(\alpha^{l-1}(x),x^l)\\
&= (\lambda + 1 - \lambda)x^{l+1} \quad\text{(by the $(l+1)$st Hom-power associativity of $A$)}\\
&= x^{l+1}.
\end{split}
\]
This finishes the induction and proves \eqref{alambdapower}.

Now suppose $2\leq l \leq k$ and $1\leq i \leq l-1$.  To prove that $A(\lambda)$ is $l$th Hom-power associative, we compute as follows:
\[
\begin{split}
&\mulambda\left(\alpha^{l-i-1}((x^i)'),\alpha^{i-1}((x^{l-i})')\right)\\
&= \lambda\mu\left(\alpha^{l-i-1}(x^i),\alpha^{i-1}(x^{l-i})\right) + (1-\lambda)\mu\left(\alpha^{i-1}(x^{l-i}),\alpha^{l-i-1}(x^i)\right)
\quad\text{(by \eqref{alambdapower})}\\
&= (\lambda + 1 - \lambda)x^l \quad\text{($l$th Hom-power associativity of $A$)}\\
&= x^l = (x^l)' \quad\text{(by \eqref{alambdapower})}.
\end{split}
\]
This shows that $A(\lambda)$ is up to $k$th Hom-power associative, as desired.
\end{proof}

The following result is an immediate consequence of Theorem \ref{thm:alambda}.

%%%%%%%%%%%%%%%%%%%%
\begin{corollary}
\label{cor1:alambda}
Let $(A,\mu,\alpha)$ be a Hom-algebra.  Then $A$ is Hom-power associative if and only if $A(\lambda)$ is Hom-power associative for all $\lambda \in \bk$.
\end{corollary}
%%%%%%%%%%%%%%%%%%%%

The next result is the special case of Corollary \ref{cor1:alambda} with $\lambda = 1/2$.

%%%%%%%%%%%%%%%%%%%%
\begin{corollary}
\label{cor2:alambda}
Let $A$ be a Hom-power associative algebra.  Then
\[
A^+ = (A,\ast = (\mu + \muop)/2,\alpha)
\]
is a commutative Hom-power associative algebra.
\end{corollary}
%%%%%%%%%%%%%%%%%%%%

%%%%%%%%%%%%%%%%%%%%%%%%%%%%%%%%%%%%%%%%%%%%
\section{Third Hom-power associativity}
\label{sec:third}
%%%%%%%%%%%%%%%%%%%%%%%%%%%%%%%%%%%%%%%%%%%%

In this section we give several characterizations of third Hom-power associativity.  Then we discuss some of their consequences, including the relationships between third Hom-power associative algebras and Hom-Lie admissible algebras.

Let us first recall the Hom-type analogue of the associator and some classes of Hom-algebras from \cite{ms}.

%%%%%%%%%%%%%%%%%%%%%%%%%
\begin{definition}
\label{def:homassociator}
Let $(A,\mu,\alpha)$ be a Hom-algebra.
\begin{enumerate}
\item
Define the \textbf{Hom-associator} $as_A \colon A^{\otimes 3} \to A$ by
\[
as_A = \mu \circ (\mu \otimes \alpha - \alpha \otimes \mu).
\]
\item
Define the \textbf{cyclic Hom-associator} $S_A \colon A^{\otimes 3} \to A$ by
\[
S_A = as_A \circ \cyclic,
\]
where $\pi$ is the cyclic permutation $(1~2~3)$.
\item
$A$ is called a \textbf{Hom-associative algebra} if $as_A = 0$.
\item
$A$ is called a \textbf{Hom-alternative algebra} if $as_A$ is totally anti-symmetric \cite{mak}.
\item
$A$ is called a \textbf{Hom-flexible algebra} if $as_A(x,y,x) = 0$ for all $x,y \in A$.
\end{enumerate}
\end{definition}
%%%%%%%%%%%%%%%%%%%%%%%%%

In terms of elements $x,y,z \in A$, we have
\begin{equation}
\label{homass}
as_A(x,y,z) = (xy)\alpha(z) - \alpha(x)(yz)
\end{equation}
and
\[
S_A(x,y,z) = as_A(x,y,z) + as_A(z,x,y) + as_A(y,z,x).
\]
When there is no danger of confusion, we will omit the subscript $A$ in the (cyclic) Hom-associator.

It is immediate from the definitions that every Hom-associative algebra is a Hom-alternative algebra, and every Hom-alternative algebra is a Hom-flexible algebra.  Also note that every commutative Hom-algebra is Hom-flexible.

The following result provides several characterizations of third Hom-power associativity and is the Hom-type analogue of \cite{albert1} (Lemma 1).  One consequence of this result (Proposition \ref{prop:n-1}) is crucial to the proof of the main result in section \ref{sec:maintheorem}.  We use the notation
\begin{equation}
\label{bullet}
x \bullet y = xy + yx
\end{equation}
for elements $x$ and $y$ in a Hom-algebra $A$.

%%%%%%%%%%%%%%%%%%%%%%%%%%
\begin{theorem}
\label{thm:third}
Let $(A,\mu,\alpha)$ be a Hom-algebra.  Then the following statements are equivalent.
\begin{enumerate}
\item
$A$ is third Hom-power associative.
\item
The condition
\begin{equation}
\label{third:a}
as_A(x,x,x) = 0
\end{equation}
holds for all $x \in A$.
\item
The condition
\begin{equation}
\label{third:b}
\sum_{\sigma \in S_3} as_A \circ\sigma = 0
\end{equation}
holds.
\item
The condition
\begin{equation}
\label{third:c}
[x \bullet y,\alpha(z)] + [z \bullet x,\alpha(y)] + [y \bullet z,\alpha(x)] = 0
\end{equation}
holds for all $x,y,z \in A$, where $[,] = \mu - \muop$ is the commutator.
\item
The cyclic Hom-associator $S_A$ is totally anti-symmetric.
\end{enumerate}
\end{theorem}
%%%%%%%%%%%%%%%%%%%%%%%%%%

\begin{proof}
Third Hom-power associativity \eqref{nhpa} says
\[
x^3 = x^2\alpha(x) = \alpha(x)x^2
\]
for all $x \in A$.  Since the first equality holds by the definition of $x^3$, third Hom-power associativity is equivalent to
\[
0 = (xx)\alpha(x) - \alpha(x)(xx) = as(x,x,x),
\]
which is the condition \eqref{third:a}.

To show tat \eqref{third:a} and \eqref{third:b} are equivalent, first observe that
\[
as(x,x,x) = \frac{1}{6} \sum_{\sigma \in S_3} as \circ \sigma (x,x,x),
\]
so \eqref{third:b} implies \eqref{third:a}.  For the converse, we linearize \eqref{third:a} by replacing $x$ with $x + \lambda w$ for $\lambda \in \bk$ and $w \in A$.  Collecting the coefficients of the powers of $\lambda$, we obtain:
\[
\begin{split}
0 &= as(x + \lambda w, x + \lambda w, x + \lambda w)\\
&= as(x,x,x) + \lambda A(x,w) + \lambda^2 A(w,x) + \lambda^3 as(w,w,w)\\
&= \lambda A(x,w) + \lambda^2 A(w,x),
\end{split}
\]
where
\[
A(x,w) = as(x,x,w) + as(x,w,x) + as(w,x,x).
\]
Setting $\lambda = \pm 1$ and subtracting the results, we infer that \eqref{third:a} implies
\begin{equation}
\label{Axw}
A(x,w) = 0
\end{equation}
for all $x,w \in A$.  Next we replace $x$ with $x + \gamma z$ for $\gamma \in \bk$ and $z \in A$ in \eqref{Axw}.  Collecting the coefficients of the powers of $\gamma$, we obtain:
\[
\begin{split}
0 &= A(x + \gamma z,w)\\
&= as(x+ \gamma z,x+ \gamma z,w) + as(x+ \gamma z,w,x+ \gamma z) + as(w,x+ \gamma z,x+ \gamma z)\\
&= A(x,w) + \gamma B(w,x,z) + \gamma^2 A(z,w)\\
&= \gamma B(w,x,z),
\end{split}
\]
where
\[
B(x_1,x_2,x_3) = \sum_{\sigma \in S_3} as \circ \sigma (x_1,x_2,x_3).
\]
Setting $\gamma = 1$, we infer that \eqref{third:a} implies $B = 0$, which is the condition \eqref{third:b}.

To see that \eqref{third:b} and \eqref{third:c} are equivalent, apply \eqref{third:b} to a generic tensor $x \otimes y \otimes z \in A^{\otimes 3}$.  Expanding the left-hand side of \eqref{third:b} using \eqref{homass}, we obtain twelve terms.  If we expand the left-hand side of \eqref{third:c} using $[,] = \mu - \muop$, we obtain the same twelve terms.  Therefore, \eqref{third:b} and \eqref{third:c} are equivalent.

Finally, observe that we have
\[
\sum_{\sigma \in S_3} as \circ\sigma = S \circ (Id + (1~2)) = S \circ (Id + (2~3)) = S \circ (Id + (1~3)).
\]
So \eqref{third:b} implies that $S$ is anti-symmetric in any two of its three variables.  Since
\[
S = S \circ \pi = S \circ \pi^2,
\]
where $\pi$ is the cyclic permutation $(1~2~3)$, we infer that \eqref{third:b} implies that $S$ is totally anti-symmetric.  On the other hand, if $S$ is totally anti-symmetric, then
\[
0 = S(x,x,x) = 3 as(x,x,x),
\]
so $as(x,x,x) = 0$, which is the condition \eqref{third:a}.  Since \eqref{third:a} is equivalent to \eqref{third:b}, the proof is complete.
\end{proof}

We now discuss some consequences of Theorem \ref{thm:third}.  First we observe that the Hom-algebras in Definition \ref{def:homassociator} are all third Hom-power associative.

%%%%%%%%%%%%%%%%%%%%%%
\begin{corollary}
\label{cor:flexible}
Every Hom-flexible algebra is third Hom-power associative.
\end{corollary}
%%%%%%%%%%%%%%%%%%%%%%

\begin{proof}
A Hom-flexible algebra $A$ satisfies $as(x,y,x) = 0$ for all $x$ and $y$ in $A$, which in particular implies $as(x,x,x) = 0$.  So $A$ is third Hom-power associative by Theorem \ref{thm:third}.
\end{proof}

The next result is the Hom-version of \cite{albert1} (Lemma 2).  It says that up to $(n-1)$st Hom-power associativity implies part of $n$th Hom-power associativity.  It is a prelude to Theorem \ref{mainthm} and will play a crucial role in the proof of that result.

%%%%%%%%%%%%%%%%%
\begin{proposition}
\label{prop:n-1}
Let $(A,\mu,\alpha)$ be a multiplicative up to $(n-1)$st Hom-power associative algebra for some $n \geq 4$.  Then
\begin{equation}
\label{commute}
[\alpha^{\lambda - 1}(x^{n-\lambda}), \alpha^{n-\lambda-1}(x^\lambda)] = 0
\end{equation}
for all $x \in A$ and $\lambda \in \{1,\ldots,n-1\}$, where $[,] = \mu - \muop$ is the commutator.
\end{proposition}
%%%%%%%%%%%%%%%%%

\begin{proof}
Pick $x \in A$ and define
\[
x(\lambda) = [\alpha^{\lambda - 1}(x^{n-\lambda}), \alpha^{n-\lambda-1}(x^\lambda)].
\]
We must show that $x(\lambda) = 0$ for $\lambda \in \{1,\ldots,n-1\}$.  First we claim that
\begin{equation}
\label{star}
x(\lambda+1) = x(\lambda) + x(1)
\end{equation}
for $\lambda \in \{1,\ldots,n-2\}$.  To prove \eqref{star}, note that, with the notation in \eqref{bullet}, the up to $(n-1)$st Hom-power associativity of $A$ implies:
\begin{equation}
\label{123}
\begin{split}
2x^{\lambda+1} &= x^\lambda \bullet \alpha^{\lambda-1}(x),\\
2x^{n-\lambda} &= x^{n-\lambda-1} \bullet \alpha^{n-\lambda-2}(x),\\
2x^{n-1} &= \alpha^{n-\lambda-2}(x^\lambda) \bullet \alpha^{\lambda-1}(x^{n-\lambda-1}).
\end{split}
\end{equation}
Applying $\alpha^{n-\lambda-2}$ and $\alpha^{\lambda-1}$ to the first and the second equalities in \eqref{123}, respectively, and using multiplicativity, we obtain:
\begin{equation}
\label{12'}
\begin{split}
2\alpha^{n-\lambda-2}(x^{\lambda+1}) &= \alpha^{n-\lambda-2}(x^\lambda) \bullet \alpha^{n-3}(x),\\
2\alpha^{\lambda-1}(x^{n-\lambda}) &= \alpha^{\lambda-1}(x^{n-\lambda-1}) \bullet \alpha^{n-3}(x).
\end{split}
\end{equation}
In the condition \eqref{third:c} in Theorem \ref{thm:third}, replace $(x,y,z)$ by $\left(\alpha^{n-\lambda-2}(x^\lambda), \alpha^{n-3}(x), \alpha^{\lambda-1}(x^{n-\lambda-1})\right)$ and use \eqref{12'} and the last equality in \eqref{123}.  The result is:
\[
\begin{split}
0 &= [2\alpha^{n-\lambda-2}(x^{\lambda+1}), \alpha^\lambda(x^{n-\lambda-1})] + [2x^{n-1}, \alpha^{n-2}(x)] + [2\alpha^{\lambda-1}(x^{n-\lambda}), \alpha^{n-\lambda-1}(x^\lambda)]\\
&= -2x(\lambda + 1) + 2x(1) + 2x(\lambda),
\end{split}
\]
from which we obtain the desired condition \eqref{star}.

From \eqref{star} it follows that
\begin{equation}
\label{starone}
x(\lambda) = \lambda x(1)
\end{equation}
for $\lambda \in \{1, \ldots, n-1\}$.  Indeed, \eqref{starone} is trivially true when $\lambda = 1$.  Then we obtain \eqref{starone} by a finite induction, using \eqref{star} in the inductive step.

Since the commutator is anti-symmetric, we have
\[
x(n-\lambda) = -x(\lambda).
\]
Combining the previous line with the $\lambda = n-1$ case of \eqref{starone}, we have
\[
-x(1) = x(n-1) = (n-1)x(1),
\]
which implies
\[
x(1) = 0.
\]
Putting this back into \eqref{starone} completes the proof of the Proposition.
\end{proof}

In the rest of this section, we discuss the relationships between third Hom-power associative algebras and Hom-Lie admissible algebras.  These results are not used in later sections.

Let us first recall the Hom-type analogues of Lie (admissible) algebras from \cite{ms}.

%%%%%%%%%%%%%%%%%%%%%%%%%
\begin{definition}
\label{def:homlie}
Let $(A,\mu,\alpha)$ be a Hom-algebra.
\begin{enumerate}
\item
Define its \textbf{Hom-Jacobian} $J_A \colon A^{\otimes 3} \to A$ as
\[
J_A = \mu \circ (\mu \otimes \alpha) \circ (Id + \pi + \pi^2),
\]
where $\pi$ is the cyclic permutation $(1~2~3)$.
\item
$A$ is called a \textbf{Hom-Lie algebra} if $\mu$ is anti-symmetric and if $J_A = 0$ (Hom-Jacobi identity).
\item
Define the \textbf{commutator Hom-algebra} as $A^- = (A,[,],\alpha)$, where $[,] = \mu - \muop$ is the commutator.
\item
$A$ is called a \textbf{Hom-Lie admissible algebra} if $A^-$ is a Hom-Lie algebra, i.e., $J_{A^-} = 0$.
\end{enumerate}
\end{definition}
%%%%%%%%%%%%%%%%%%%%%%%%%

One way to understand Hom-Lie admissible algebras is through the classes of $G$-Hom-associative algebras, which we now recall from \cite{ms}.  The classes of $G$-associative algebras were first defined in \cite{gm}.

%%%%%%%%%%%%%%%%%%%
\begin{definition}
\label{def:G}
Let $(A,\mu,\alpha)$ be a Hom-algebra and $G$ be a subgroup of the symmetric group $S_3$.  Then $A$ is called a \textbf{$G$-Hom-associative algebra} if it satisfies
\begin{equation}
\label{gass}
\sum_{\sigma \in G} \sign as_A \circ \sigma = 0,
\end{equation}
where $as_A$ is the Hom-associator (Definition \ref{def:homassociator}).
\end{definition}
%%%%%%%%%%%%%%%%%%%

Clearly Hom-associative algebras are exactly the $\{Id\}$-Hom-associative algebras.  Let $A_3$ be the subgroup $\{Id,\pi,\pi^2\}$ of $S_3$ generated by the cyclic permutation $\pi = (1~2~3)$.  Then Hom-Lie algebras are exactly the $A_3$-Hom-associative algebras whose multiplications are anti-symmetric.  Moreover, as shown in \cite{ms}, every $G$-Hom-associative algebra is a Hom-Lie admissible algebra.  Conversely, a Hom-Lie admissible algebra is an $S_3$-Hom-associative algebra.  Examples of and construction results along the lines of Corollary \ref{cor2:twist} for $G$-Hom-associative algebras can be found in \cite{yau2}.

It follows from Corollary \ref{cor:flexible} that Hom-associative algebras are all third Hom-power associative.  So some Hom-Lie admissible algebras are third Hom-power associative.  This raises the question:
\begin{quote}
Exactly which Hom-Lie admissible algebras are third Hom-power associative?
\end{quote}
We answer this question in Theorem \ref{thm:a3}.  The answer turns out to be the class of $A_3$-Hom-associative algebras.

The following two preliminary results are used in the proof of Theorem \ref{thm:a3}.  The next result describes $A_3$-Hom-associative algebras in terms of its cyclic Hom-associator.

%%%%%%%%%%%%%%%%%%%%
\begin{lemma}
\label{lem1:A3}
Let $(A,\mu,\alpha)$ be a Hom-algebra.  Then $A$ is an $A_3$-Hom-associative algebra if and only if $S_A = 0$, where $S_A$ is the cyclic Hom-associator.
\end{lemma}
%%%%%%%%%%%%%%%%%%%%

\begin{proof}
Every element in the subgroup $A_3 = \{Id,\pi,\pi^2\}$ of $S_3$ is even.  So we have
\[
\sum_{\sigma \in A_3} \sign as_A \circ \sigma = as_A \circ (Id + \pi + \pi^2) = S_A,
\]
which proves the Lemma.
\end{proof}

The next result describes the Hom-Jacobian of the commutator Hom-algebra in terms of the cyclic Hom-associator.

%%%%%%%%%%%%%%%%%%%%
\begin{lemma}
\label{lem2:A3}
Let $(A,\mu,\alpha)$ be a Hom-algebra and $A^- = (A,[,],\alpha)$ be its commutator Hom-algebra.  Then
\[
J_{A^-} = S_A \circ (Id - (2~3)),
\]
where $J_{A^-}$ is the Hom-Jacobian of $A^-$.
\end{lemma}
%%%%%%%%%%%%%%%%%%%%

\begin{proof}
Expanding in terms of $\mu$, a simple calculation shows that the twelve terms in $J_{A^-}$ are exactly the same as the twelve terms in $S_A \circ (Id - (2~3))$.
\end{proof}

The following result says that, within the class of Hom-algebras, the intersection of the classes of third Hom-power associative algebras and of Hom-Lie admissible algebras is equal to the class of $A_3$-Hom-associative algebras.

%%%%%%%%%%%%%%%%
\begin{theorem}
\label{thm:a3}
Let $(A,\mu,\alpha)$ be a Hom-algebra.  Then $A$ is an $A_3$-Hom-associative algebra if and only if it is both a third Hom-power associative algebra and a Hom-Lie admissible algebra.
\end{theorem}
%%%%%%%%%%%%%%%%

\begin{proof}
First assume that $A$ is an $A_3$-Hom-associative algebra.  Then it is Hom-Lie admissible \cite{ms}.  Moreover, by Lemma \ref{lem1:A3} we have $S_A = 0$.  In particular, $S_A$ is totally anti-symmetric, so Theorem \ref{thm:third} implies that $A$ is third Hom-power associative.

Conversely, assume that $A$ is both a third Hom-power associative algebra and a Hom-Lie admissible algebra.  Since $A$ is third Hom-power associative, by Theorem \ref{thm:third}, $S_A$ is totally anti-symmetric, which implies
\[
S_A = -S_A \circ (2~3).
\]
The previous line and Lemma \ref{lem2:A3} imply
\[
S_A = \frac{1}{2}\left(S_A - S_A \circ (2~3)\right)
= \frac{1}{2}J_{A^-}.
\]
Since $A$ is Hom-Lie admissible, we have $J_{A^-} = 0$, and hence $S_A = 0$. Lemma \ref{lem1:A3} now says that $A$ is $A_3$-Hom-associative.
\end{proof}

%%%%%%%%%%%%%%%%%%%%%%%%%%%%%%%%%%%%%%%%%%%%
\section{Fourth Hom-power associativity}
\label{sec:fourth}
%%%%%%%%%%%%%%%%%%%%%%%%%%%%%%%%%%%%%%%%%%%%

The purpose of this section is to study fourth Hom-power associativity.  Using certain Hom-alternative and Hom-Jordan algebras, we also demonstrate that in a Hom-power associative algebra, it is not necessarily the case that every element generates a Hom-associative sub-Hom-algebra.

Let us first record some consequences of third Hom-power associativity regarding fourth Hom-power associativity.  The following result is the special case of Proposition \ref{prop:n-1} with $n=4$ and $\lambda=1$.  It says that up to third Hom-power associativity implies part of fourth Hom-power associativity.

%%%%%%%%%%%%%%%%%%%
\begin{corollary}
\label{cor1:third}
Let $(A,\mu,\alpha)$ be a multiplicative third Hom-power associative algebra.  Then
\begin{equation}
\label{eq:cor1third}
x^4 = x^3\alpha^2(x) = \alpha^2(x)x^3
\end{equation}
for all $x \in A$.
\end{corollary}
%%%%%%%%%%%%%%%%%%%

The following consequence of Corollary \ref{cor1:third} provides a characterization of up to fourth Hom-power associativity in terms of two identities.

%%%%%%%%%%%%%%%%%%%
\begin{corollary}
\label{cor2:third}
Let $(A,\mu,\alpha)$ be a multiplicative Hom-algebra.  Then the following statements are equivalent.
\begin{enumerate}
\item
$A$ satisfies
\begin{equation}
\label{uptofourth}
x^2\alpha(x) = \alpha(x)x^2 \quad\text{and}\quad
x^4 = \alpha(x^2)\alpha(x^2)
\end{equation}
for all $x \in A$.
\item
$A$ satisfies
\begin{equation}
\label{uptofourth'}
as(x,x,x) =  0 = as(x^2,\alpha(x),\alpha(x))
\end{equation}
for all $x \in A$.
\item
$A$ is up to fourth Hom-power associative.
\end{enumerate}
\end{corollary}
%%%%%%%%%%%%%%%%%%%

\begin{proof}
The identity $as(x,x,x) = 0$ is equivalent to the first identity in \eqref{uptofourth} by the definition of the Hom-associator.  Moreover, since $A$ is multiplicative, the identity $as(x^2,\alpha(x),\alpha(x)) = 0$ is equivalent to:
\[
\begin{split}
0 &= as(x^2,\alpha(x),\alpha(x))\\
&= (x^2\alpha(x))\alpha^2(x) - \alpha(x^2)(\alpha(x)\alpha(x))\\
&= x^4 - \alpha(x^2)\alpha(x^2).
\end{split}
\]
Therefore, \eqref{uptofourth} and \eqref{uptofourth'} are equivalent.

Next we show the equivalence between the first and the last statements.  The first equality in \eqref{uptofourth} is third Hom-power associativity, and the second equality is part of fourth Hom-power associativity.  Conversely, suppose \eqref{uptofourth} is true.  Then $A$ is multiplicative third Hom-power associative.  By Corollary \ref{cor1:third} we have
\[
x^4 = x^3\alpha^2(x) = \alpha(x^2)\alpha(x^2) = \alpha^2(x)x^3,
\]
which means that $A$ is fourth Hom-power associative.
\end{proof}

In Corollary \ref{cor1:main} below, we will see that the Hom-algebras in Corollary \ref{cor2:third} are, in fact, Hom-power associative.

Next we discuss some classes of fourth Hom-power associative algebras.

%%%%%%%%%%%%%%%%%%%%%%
\begin{definition}
\label{def:homalt}
Let $(A,\mu,\alpha)$ be a Hom-algebra.
\begin{enumerate}
\item
$A$ is called a \textbf{right Hom-alternative algebra} \cite{mak} if it satisfies
\begin{equation}
\label{homalt}
as_A(y,x,x) = 0
\end{equation}
for all $x,y \in A$.
\item
$A$ is called a \textbf{non-commutative Hom-Jordan algebra} if $A$ is Hom-flexible, i.e., $as(x,y,x) = 0$, and if
\begin{equation}
\label{homjordan}
as_A(x^2,\alpha(y),\alpha(x)) = 0
\end{equation}
for all $x,y \in A$.
\item
$A$ is called a \textbf{Hom-Jordan algebra} \cite{yau12} if $A$ is commutative and satisfies \eqref{homjordan} for all $x,y \in A$.
\end{enumerate}
\end{definition}
%%%%%%%%%%%%%%%%%%%%%%%

Construction results for these Hom-algebras along the lines of Theorem \ref{thm:twist} and Corollaries \ref{cor1:twist} and \ref{cor2:twist} can be easily formulated and proved.  Right-alternative and (non-commutative) Jordan algebras are equivalent to right Hom-alternative and (non-commutative) Hom-Jordan algebras with identity twisting maps.  The conditions \eqref{homalt} and \eqref{homjordan} are called the \textbf{right Hom-alternative identity} and the \textbf{Hom-Jordan identity}, respectively.  They are the Hom-type generalizations of the usual right-alternative identity and the Jordan identity.

Every Hom-alternative algebra is a right Hom-alternative algebra.  Every Hom-Jordan algebra is a non-commutative Hom-Jordan algebra, since commutativity implies Hom-flexibility.  Analogous to the classical case, it is proved in \cite{yau12} that multiplicative Hom-alternative algebras are both Hom-Jordan admissible and Hom-Maltsev admissible.  We should point out that the Hom-Jordan identity above is different from the one defined in \cite{mak}.

%%%%%%%%%%%%%%%%%%%%%%%
\begin{proposition}
\label{prop:homjordan}
Let $(A,\mu,\alpha)$ be a multiplicative Hom-algebra that is either a right Hom-alternative algebra or a non-commutative Hom-Jordan algebra.  Then $A$ is up to fourth Hom-power associative.
\end{proposition}
%%%%%%%%%%%%%%%%%%%%%%%

\begin{proof}
By Corollary \ref{cor2:third} we only need to show that $A$ satisfies the two identities in \eqref{uptofourth'}.  The identity $as(x,x,x) = 0$ is satisfied in both cases because of either right Hom-alternativity \eqref{homalt} or Hom-flexibility.  The other identity in \eqref{uptofourth'}, $as(x^2,\alpha(x),\alpha(x)) = 0$,  can be obtained from either (i) the right Hom-alternative identity \eqref{homalt} by replacing $(x,y)$ by $(\alpha(x),x^2)$, or from (ii) the Hom-Jordan identity \eqref{homjordan} by setting $y = x$.
\end{proof}

In Corollary \ref{cor3:main} below, we will see that the Hom-algebras in Proposition \ref{prop:homjordan} are, in fact, Hom-power associative.  Using Hom-alternative and Hom-Jordan algebras, let us point out a major conceptual difference between power associative algebras and Hom-power associative algebras.

Recall from the introduction that a power associative algebra can be equivalently defined as an algebra in which every element generates an associative sub-algebra.  In contrast, in a general Hom-power associative algebra, it is not true that every element generates a Hom-associative sub-Hom-algebra.  This latter condition is too restrictive, and it excludes some Hom-alternative and Hom-Jordan algebras, all of which are Hom-power associative by Corollary \ref{cor3:main}.  In fact, in a Hom-algebra $(A,\mu,\alpha)$, the sub-Hom-algebra generated by an element $x$ includes $\alpha^k(x)$ for all $k \geq 1$.  However, the Hom-associator (Definition \ref{def:homassociator})
\begin{equation}
\label{xkl}
as_A(x,\alpha(x),\alpha^2(x)) = (x\alpha(x))\alpha^{3}(x) - \alpha(x)(\alpha(x)\alpha^2(x))
\end{equation}
does not need to be equal to $0$ in a Hom-alternative algebra or a Hom-Jordan algebra.  When this particular Hom-associator is not equal to $0$, the sub-Hom-algebra generated by $x$ is not a Hom-associative algebra.  We illustrate this in the following two examples.

%%%%%%%%%%%%%%%%
\begin{example}
\label{ex:oct}
In this example, we exhibit a multiplicative Hom-alternative algebra (hence a Hom-power associative algebra by Corollary \ref{cor3:main}) in which the Hom-associator in \eqref{xkl} is not equal to $0$ for some element $x$.

The octonions $\oct$ \cite{baez} is constructed from the field $\bk$ of real numbers using the Cayley-Dickson process three times with the parameter $-1$ in each step.  More concretely, the octonions $\oct$ is an eight-dimensional alternative (but not associative) algebra with a basis $\{e_0,\ldots,e_7\}$ and the following multiplication table, where $\mu$ denotes the multiplication in $\oct$.
\begin{center}
\begin{tabular}{|c|c|c|c|c|c|c|c|c|}\hline
$\mu$ & $e_0$ & $e_1$ & $e_2$ & $e_3$ & $e_4$ & $e_5$ & $e_6$ & $e_7$ \\\hline
$e_0$ & $e_0$ & $e_1$ & $e_2$ & $e_3$ & $e_4$ & $e_5$ & $e_6$ & $e_7$ \\\hline
$e_1$ & $e_1$ & $-e_0$ & $e_4$ & $e_7$ & $-e_2$ & $e_6$ & $-e_5$ & $-e_3$ \\\hline
$e_2$ & $e_2$ & $-e_4$ & $-e_0$ & $e_5$ & $e_1$ & $-e_3$ & $e_7$ & $-e_6$ \\\hline
$e_3$ & $e_3$ & $-e_7$ & $-e_5$ & $-e_0$ & $e_6$ & $e_2$ & $-e_4$ & $e_1$ \\\hline
$e_4$ & $e_4$ & $e_2$ & $-e_1$ & $-e_6$ & $-e_0$ & $e_7$ & $e_3$ & $-e_5$ \\\hline
$e_5$ & $e_5$ & $-e_6$ & $e_3$ & $-e_2$ & $-e_7$ & $-e_0$ & $e_1$ & $e_4$ \\\hline
$e_6$ & $e_6$ & $e_5$ & $-e_7$ & $e_4$ & $-e_3$ & $-e_1$ & $-e_0$ & $e_2$ \\\hline
$e_7$ & $e_7$ & $e_3$ & $e_6$ & $-e_1$ & $e_5$ & $-e_4$ & $-e_2$ & $-e_0$ \\\hline
\end{tabular}
\end{center}
There is an algebra automorphism $\alpha \colon \oct \to \oct$ given by
\begin{equation}
\label{octaut}
\begin{split}
\alpha(e_0) = e_0,\quad \alpha(e_1) = e_5,\quad \alpha(e_2) = e_6,\quad \alpha(e_3) = e_7,\\
\alpha(e_4) = e_1,\quad \alpha(e_5) = e_2,\quad \alpha(e_6) = e_3,\quad \alpha(e_7) = e_4.
\end{split}
\end{equation}
Using \cite{mak} (Theorem 3.1), which is the analogue of Corollary \ref{cor2:twist} for Hom-alternative algebras, we obtain a multiplicative Hom-alternative algebra
\[
\octalpha = (\oct,\mualpha = \alpha\mu,\alpha)
\]
with the following multiplication table.
\begin{center}
\begin{tabular}{|c|c|c|c|c|c|c|c|c|}\hline
$\mualpha$ & $e_0$ & $e_1$ & $e_2$ & $e_3$ & $e_4$ & $e_5$ & $e_6$ & $e_7$ \\\hline
$e_0$ & $e_0$ & $e_5$ & $e_6$ & $e_7$ & $e_1$ & $e_2$ & $e_3$ & $e_4$\\\hline
$e_1$ & $e_5$ & $-e_0$ & $e_1$ & $e_4$ & $-e_6$ & $e_3$ & $-e_2$ & $-e_7$\\\hline
$e_2$ & $e_6$ & $-e_1$ & $-e_0$ & $e_2$ & $e_5$ & $-e_7$ & $e_4$ & $-e_3$\\\hline
$e_3$ & $e_7$ & $-e_4$ & $-e_2$ & $-e_0$ & $e_3$ & $e_6$ & $-e_1$ & $e_5$\\\hline
$e_4$ & $e_1$ & $e_6$ & $-e_5$ & $-e_3$ & $-e_0$ & $e_4$ & $e_7$ & $-e_2$\\\hline
$e_5$ & $e_2$ & $-e_3$ & $e_7$ & $-e_6$ & $-e_4$ & $-e_0$ & $e_5$ & $e_1$\\\hline
$e_6$ & $e_3$ & $e_2$ & $-e_4$ & $e_1$ & $-e_7$ & $-e_5$ & $-e_0$ & $e_6$\\\hline
$e_7$ & $e_4$ & $e_7$ & $e_3$ & $-e_5$ & $e_2$ & $-e_1$ & $-e_6$ & $-e_0$\\\hline
\end{tabular}
\end{center}
By Corollary \ref{cor3:main} below, $\octalpha$ is a Hom-power associative algebra.  On the other hand, one can check that
\[
as_{\octalpha}(e_1,\alpha(e_1),\alpha^2(e_1)) = -2e_1 \not= 0.
\]
Therefore, the sub-Hom-algebra in the Hom-power associative algebra $\octalpha$ generated by $e_1$ is not Hom-associative.
\qed
\end{example}
%%%%%%%%%%%%%%%%

%%%%%%%%%%%%%%%%%%%%%
\begin{example}
\label{ex:jordan}
In this example, we exhibit a multiplicative Hom-Jordan algebra (hence a Hom-power associative algebra by Corollary \ref{cor3:main}) in which the Hom-associator in \eqref{xkl} is not equal to $0$ for some element $x$.

Consider the $27$-dimensional exceptional simple Jordan algebra $M^8_3$ \cite{schafer}.  The elements of $M^8_3$ are $3 \times 3$ Hermitian octonionic matrices, i.e., matrices of the form
\[
X =
\begin{pmatrix}
a_1 & x & y\\
\xbar & a_2 & z\\
\ybar & \zbar & a_3
\end{pmatrix}
\]
with each $a_i \in \bk$ and $x,y,z \in \oct$ (see Example \ref{ex:oct}).  Here we are using the convention $a_i = a_ie_0$ for the diagonal elements.  For an octonion $x = \sum_{i=0}^7 b_ie_i$ with each $b_i \in \bk$, its \emph{conjugate} is defined as the octonion\[
\xbar = b_0e_0 - \sum_{i=1}^7 b_ie_i.
\]
This $\bk$-module $M^8_3$ becomes a Jordan algebra with the multiplication
\[
X \ast Y = \frac{1}{2}(XY + YX),
\]
where $XY$ and $YX$ are the usual matrix multiplication.

Let $\alpha \colon \oct \to \oct$ be the algebra automorphism in \eqref{octaut}, which is both unit-preserving and conjugate-preserving, i.e., $\alpha(e_0) = e_0$ and $\alpha(\xbar) = \overline{\alpha(x)}$ for all $x \in \oct$.  It extends entrywise to a linear automorphism $\alpha \colon M^8_3 \to M^8_3$.  This extended map $\alpha$ respects matrix multiplication and the Jordan product $\ast$, i.e., $\alpha$ is an algebra automorphism on $(M^8_3,\ast)$.  Using the Hom-Jordan analogue of Corollary \ref{cor2:twist} (= Theorem 5.8 in \cite{yau12}), we obtain a multiplicative Hom-Jordan algebra
\[
(M^8_3)_\alpha = (M^8_3,\ast_\alpha = \alpha \circ \ast,\alpha).
\]
By Corollary \ref{cor3:main} below, $(M^8_3)_\alpha$ is a Hom-power associative algebra.  On the other hand, one can check that
\[
as_{(M^8_3)_\alpha}(X,\alpha(X),\alpha^2(X)) \not= 0
\]
for the element
\[
X = \matrixy.
\]
Therefore, the sub-Hom-algebra in the Hom-power associative algebra $(M^8_3)_\alpha$ generated by $X$ is not Hom-associative.
\qed
\end{example}
%%%%%%%%%%%%%%%%%%%%%

The rest of this section is devoted to studying the identity $x^4 = \alpha(x^2)\alpha(x^2)$, which is part of fourth Hom-power associativity.
To state the result below, we introduce the following notations to represent the linearized forms of $x^4 = ((xx)\alpha(x))\alpha^2(x)$ and $\alpha(x^2)\alpha(x^2) = \alpha(xx)\alpha(xx)$.

%%%%%%%%%%%%%%%%%%
\begin{definition}
\label{def:F}
Let $(A,\mu,\alpha)$ be a Hom-algebra, and let $x,y,z$, and $w$ be elements in $A$. Regard the permutations in $S_4$ as permutations of the letters $(x,y,z,w)$.  Consider the subsets
\[
\begin{split}
I_L &= \{(1~2~3)^i(1~2~3~4)^j \colon 0 \leq i \leq 2,\, 0 \leq j \leq 3\},\\
I_R &= \{Id,\, (1~3),\, (2~3),\, (1~4),\, (2~4),\, (1~3)(2~4)\}
\end{split}
\]
of $S_4$.  Let $\sum_L$ and $\sum_R$ denote the sums indexed by the sets $I_L$ and $I_R$, respectively.  Define the elements
\[
\begin{split}
F_L(x,y,z,w) &= \sum_L \left((x \bullet y)\alpha(z)\right)\alpha^2(w),\\
F_R(x,y,z,w) &= \sum_R \alpha(x \bullet y)\alpha(z \bullet w),\\
F(x,y,z,w) &= F_L(x,y,z,w) - F_R(x,y,z,w)
\end{split}
\]
in $A$, where $x \bullet y = xy + yx$ as in \eqref{bullet}.
\end{definition}
%%%%%%%%%%%%%%%%%%

Note that $x \bullet y = y \bullet x$ and that there are twelve terms in $\sum_L$ and six terms in $\sum_R$.  The twelve terms in $\sum_L$ are the twelve permutations of the four variables in the expression $\left((x \bullet y)\alpha(z)\right)\alpha^2(w)$.  Likewise, the six terms in $\sum_R$ are the six permutations of the four variables in the expression $\alpha(x \bullet y)\alpha(z \bullet w)$.

The following result is the Hom-version of \cite{albert1} (Lemma 3).  It gives a linearized form of the identity $x^4 = \alpha(x^2)\alpha(x^2)$.  It is also a crucial ingredient in the proof of Theorem \ref{mainthm}.

%%%%%%%%%%%%%%%%%%%
\begin{theorem}
\label{thm:fourth}
Let $(A,\mu,\alpha)$ be a Hom-algebra.  Then $A$ satisfies
\[
x^4 = \alpha(x^2)\alpha(x^2)
\]
for all $x \in A$ if and only if
\begin{equation}
\label{F0}
F(x,y,z,w) = 0
\end{equation}
for all $x,y,z,w \in A$.
\end{theorem}
%%%%%%%%%%%%%%%%%%%

\begin{proof}
For $x \in A$, let
\[
B(x) = (x^2\alpha(x))\alpha^2(x) - \alpha(x^2)\alpha(x^2).
\]
The ``if" part follows from the identity
\[
F(x,x,x,x) = 24B(x)
\]
and the definition $x^4 = (x^2\alpha(x))\alpha^2(x)$.

For the ``only if" part, assume that $B(x) = 0$ for all $x \in A$.  For $\lambda \in \bk$ and $y \in A$, we have
\begin{equation}
\label{A(x)}
\begin{split}
0 &= B(x + \lambda y)\\
&= B(x) + \lambda C(x,y) + \lambda^2 D(x,y) + \lambda^3 C(y,x) + \lambda^4 B(y)\\
&= \lambda C(x,y) + \lambda^2 D(x,y) + \lambda^3 C(y,x).
\end{split}
\end{equation}
Here $C(x,y)$ is some expression involving $x$, $y$, and $\alpha$, and $D(x,y)$ is
\[
\begin{split}
D(x,y) &= (y^2\alpha(x))\alpha^2(x) + (x^2\alpha(y))\alpha^2(y)\\
&\relphantom{} + \left((x \bullet y)\alpha(x)\right)\alpha^2(y)
+ \left((x \bullet y)\alpha(y)\right)\alpha^2(x) \\
&\relphantom{} - \alpha(x \bullet y)\alpha(x \bullet y) - \alpha(x^2) \bullet \alpha(y^2).
\end{split}
\]
Substituting $\lambda = \pm 1$ in \eqref{A(x)} and adding the results, one infers that $D(x,y) = 0$ for all $x,y \in A$.

Next, for $\eta \in \bk$ and $z \in A$, we have
\[
\begin{split}
0 &= D(x + \eta z,y)\\
&= D(x,y) + \eta E(x,y,z) + \eta^2 D(z,y)\\
&= \eta E(x,y,z).
\end{split}
\]
Setting $\eta = 1$, this implies $E(x,y,z) = 0$ for all $x,y,z \in A$.  A simple calculation shows that
\[
\begin{split}
E(x,y,z)
&= \left(y^2\alpha(x)\right)\alpha^2(z) + \left(y^2\alpha(z)\right)\alpha^2(x) + \left((x \bullet z)\alpha(y)\right)\alpha^2(y)\\
&\relphantom{} + \left((x \bullet y)\alpha(z)\right)\alpha^2(y) + \left((z \bullet y)\alpha(x)\right)\alpha^2(y)\\
&\relphantom{} + \left((x \bullet y)\alpha(y)\right)\alpha^2(z) + \left((z \bullet y)\alpha(y)\right)\alpha^2(x)\\
&\relphantom{} - \alpha(x \bullet y) \bullet \alpha(z \bullet y) - \alpha(x \bullet z) \bullet \alpha(y^2).
\end{split}
\]
Finally, for $\nu \in \bk$ and $w \in A$, we have
\[
\begin{split}
0 &= E(x, y+\nu w, z)\\
&= E(x,y,z) + \nu F(x,y,z,w) + \nu^2 E(x,w,z)\\
&= \nu F(x,y,z,w),
\end{split}
\]
where $F(x,y,z,w)$ is as defined in Definition \ref{def:F}.  Setting $\nu = 1$, this implies $F(x,y,z,w) = 0$ for all $x,y,z,w \in A$, as desired.
\end{proof}

%%%%%%%%%%%%%%%%%%%%%%%%%%%%%%%%%%%%%%%%%%%%
\section{Hom-power associativity}
\label{sec:maintheorem}
%%%%%%%%%%%%%%%%%%%%%%%%%%%%%%%%%%%%%%%%%%%%

The purpose of this section is to show that, in the presence of multiplicativity, up to fourth Hom-power associativity implies Hom-power associativity, generalizing a well-known result of Albert \cite{albert1} about power associative algebras.  The desired result is a consequence of the following result, which is the Hom-version of \cite{albert1} (Lemma 4).

%%%%%%%%%%%%%%%%
\begin{theorem}
\label{mainthm}
Let $(A,\mu,\alpha)$ be a multiplicative up to $(n-1)$st Hom-power associative algebra for some $n \geq 5$.  Then $A$ is up to $n$th Hom-power associative.
\end{theorem}
%%%%%%%%%%%%%%%%

Before we give the proof of Theorem \ref{mainthm}, let us record some of its consequences.

%%%%%%%%%%%%%%%%%%%%%%
\begin{corollary}
\label{cor1:main}
Let $(A,\mu,\alpha)$ be a multiplicative Hom-algebra.  Then the following statements are equivalent.
\begin{enumerate}
\item
$A$ satisfies
\[
x^2\alpha(x) = \alpha(x)x^2 \quad\text{and}\quad
x^4 = \alpha(x^2)\alpha(x^2)
\]
for all $x \in A$.
\item
$A$ satisfies
\[
as(x,x,x) =  0 = as(x^2,\alpha(x),\alpha(x))
\]
for all $x \in A$.
\item
$A$ is up to fourth Hom-power associative.
\item
$A$ is Hom-power associative.
\end{enumerate}
\end{corollary}
%%%%%%%%%%%%%%%%%%%%%%

\begin{proof}
The equivalence of the first three statements is Corollary \ref{cor2:third}.  On the other hand, if $A$ is multiplicative up to fourth Hom-power associative, then we apply Theorem \ref{mainthm} repeatedly to conclude that $A$ is up to $n$th Hom-power associative for all $n \geq 5$, i.e., that $A$ is Hom-power associative.
\end{proof}

Setting $\alpha = Id$ in Corollary \ref{cor1:main}, we recover the following result of Albert \cite{albert1} (Theorem 2).

%%%%%%%%%%%%%%%%%%%%%
\begin{corollary}
\label{cor2:main}
An algebra $A$ is power associative if and only if it satisfies
\[
x^2x = xx^2 \quad\text{and}\quad (x^2x)x = x^2x^2
\]
for all $x \in A$.
\end{corollary}
%%%%%%%%%%%%%%%%%%%%%

In the presence of Hom-flexibility (e.g., when the multiplication is commutative), Hom-power associativity requires only one identity.

%%%%%%%%%%%%%%%%%%%%%
\begin{corollary}
\label{cor1.5:main}
Let $(A,\mu,\alpha)$ be a multiplicative Hom-flexible algebra.  Then $A$ is Hom-power associative if and only if it satisfies
\[
x^4 = \alpha(x^2)\alpha(x^2)
\]
for all $x \in A$.
\end{corollary}
%%%%%%%%%%%%%%%%%%%%%

\begin{proof}
This is an immediate consequence of Corollaries \ref{cor:flexible} and \ref{cor1:main}.
\end{proof}

Setting $\alpha = Id$ in Corollary \ref{cor1.5:main}, we recover the following result of Albert \cite{albert1} (Theorem 1).

%%%%%%%%%%%%%%%%%%%%%
\begin{corollary}
\label{cor2.5:main}
Suppose $A$ is a flexible (e.g., commutative) algebra.  Then $A$ is power associative if and only if it satisfies
\[
(x^2x)x = x^2x^2
\]
for all $x \in A$.
\end{corollary}
%%%%%%%%%%%%%%%%%%%%%

Recall the two types of Hom-algebras defined in Definition \ref{def:homalt}.

%%%%%%%%%%%%%%%%%%%%
\begin{corollary}
\label{cor3:main}
Multiplicative right Hom-alternative and multiplicative non-commutative Hom-Jordan algebras are Hom-power associative.
\end{corollary}
%%%%%%%%%%%%%%%%%%%%

\begin{proof}
This is an immediate consequence of Proposition \ref{prop:homjordan} and Corollary \ref{cor1:main}.
\end{proof}

Setting the twisting map to be the identity map in Corollary \ref{cor3:main}, we obtain the following well-known result for ordinary algebras \cite{schafer}.

%%%%%%%%%%%%%%%%%%
\begin{corollary}
\label{cor4:main}
Right-alternative algebras and non-commutative Jordan algebras are power associative.
\end{corollary}
%%%%%%%%%%%%%%%%%%

The rest of this section is devoted to the proof of Theorem \ref{mainthm}, which involves a series of Lemmas.  To simplify the presentation, we use the notation
\begin{equation}
\label{xij}
x^{i,j} = \alpha^{j-1}(x^i) \alpha^{i-1}(x^j)
\end{equation}
for an element $x$ in a Hom-algebra $(A,\mu,\alpha)$ and positive integers $i$ and $j$, where $x^i$ and $x^j$ are the $i$th and the $j$th Hom-powers \eqref{hompower}.  With this notation, $n$th Hom-power associativity \eqref{nhpa} can be stated as
\begin{equation}
\label{nhpa'}
x^n = x^{n-i,i}
\end{equation}
for all elements $x$ and $1\leq i \leq n-1$.  The equality \eqref{nhpa'} is exactly what we need to prove for Theorem \ref{mainthm}.

Several equalities are used repeatedly below.  In particular, in Theorem \ref{mainthm} the assumption of up to $(n-1)$st Hom-power associativity of $A$ means
\[
x^m = x^{m-i,i}
\]
for all $x \in A$, $2 \leq m < n$, and $1 \leq i \leq m-1$.  Also, by the definition of the Hom-powers we have
\[
x^k = x^{k-1}\alpha^{k-2}(x) = x^{k-1,1}
\]
for all $k \geq 2$.  Moreover, the conclusion of Proposition \ref{prop:n-1} can be stated as
\[
x^{n-i,i} = x^{i,n-i}
\]
for all $x \in A$ and $1 \leq i \leq n-1$.

In the following Lemmas, we derive identities involving $x^n$ and $x^{n-i,i}$ for $2 \leq i \leq 5$.  Then we play these identities against each other to obtain the $n$th Hom-power associativity of $A$.  Our first objective is to prove the special case of Theorem \ref{mainthm} when $n=5$, which requires the following result.

%%%%%%%%%%%%%%%%%%
\begin{lemma}
\label{lem1:main}
Let $(A,\mu,\alpha)$ be a multiplicative up to $(n-1)$st Hom-power associative algebra for some $n \geq 5$.  Then
\begin{equation}
\label{i}
x^{n-3,3} = 4x^{n-2,2} - 3x^n
\end{equation}
for all $x \in A$.
\end{lemma}
%%%%%%%%%%%%%%%%%%

\begin{proof}
Pick $x \in A$.  The desired identity \eqref{i} is equivalent to the following special case of \eqref{F0}:
\begin{equation}
\label{Fi}
F\left(\alpha^{n-4}(x),\alpha^{n-4}(x),\alpha^{n-4}(x),x^{n-3}\right) = 0.
\end{equation}
Indeed, using the assumptions on $A$, the definition \eqref{hompower} of the Hom-powers, the notations above and those in Definition \ref{def:F}, and the fact that three of the four arguments in $F$ in \eqref{Fi} are the same, we have:
\[
\begin{split}
& F_L\left(\alpha^{n-4}(x), \alpha^{n-4}(x), \alpha^{n-4}(x), x^{n-3}\right)\\
&= 3\left((2\alpha^{n-4}(x)\alpha^{n-4}(x))\alpha^{n-3}(x)\right) \alpha^2(x^{n-3}) + 3\left((2\alpha^{n-4}(x)\alpha^{n-4}(x))\alpha(x^{n-3})\right) \alpha^{n-2}(x)\\
&\relphantom{} + 3\left(2(x^{n-3} \bullet \alpha^{n-4}(x)) \alpha^{n-3}(x)\right) \alpha^{n-2}(x)\\
&= 6x^{3,n-3} + 6\left(x^{2,n-3} + 2x^{n-2,1}\right) \alpha^{n-2}(x)\\
&= 6x^{n-3,3} + 6(3x^{n-1})\alpha^{n-2}(x)\\
&= 6x^{n-3,3} + 18x^n.
\end{split}
\]
A similar calculation yields:
\[
\begin{split}
& F_R\left(\alpha^{n-4}(x), \alpha^{n-4}(x), \alpha^{n-4}(x), x^{n-3}\right)\\
&= 3\alpha\left(2\alpha^{n-4}(x)\alpha^{n-4}(x)\right) \alpha\left(\alpha^{n-4}(x) \bullet x^{n-3}\right) + 3\alpha\left(\alpha^{n-4}(x) \bullet x^{n-3}\right) \alpha\left(2\alpha^{n-4}(x)\alpha^{n-4}(x)\right)\\
&= 12x^{2,n-2} + 12x^{n-2,2}\\
&= 24x^{n-2,2}.
\end{split}
\]
Therefore, \eqref{Fi} is equivalent to
\[
6x^{n-3,3} + 18x^n = 24x^{n-2,2},
\]
which in turn is equivalent to the desired equality \eqref{i}.
\end{proof}

Using Lemma \ref{lem1:main} we now prove the special case of Theorem \ref{mainthm} when $n=5$.

%%%%%%%%%%%%%%%%%%
\begin{lemma}
\label{lem2:main}
Let $A$ be a multiplicative up to fourth Hom-power associative algebra.  Then $A$ is up to fifth Hom-power associative.
\end{lemma}
%%%%%%%%%%%%%%%%%%

\begin{proof}
To show that $A$ is fifth Hom-power associative, pick $x \in A$.  By the $n=5$ case of \eqref{commute} we have
\[
x^5 = x^{4,1} = x^{1,4} \quad\text{and}\quad
x^{3,2} = x^{2,3}.
\]
The $n=5$ case of Lemma \ref{lem1:main} says that
\[
x^{2,3} = 4x^{3,2} - 3x^5.
\]
This implies that
\[
3x^5 = 4x^{3,2} - x^{2,3} = 3x^{3,2},
\]
and hence
\[
x^5 = x^{3,2} = x^{2,3}.
\]
We conclude that
\[
x^5 = x^{i,5-i}
\]
for $1 \leq i \leq 4$, so $A$ is fifth Hom-power associative, as desired.
\end{proof}

The following result relates $x^{n-4,4}$ to $x^{n-2,2}$ and $x^n$.

%%%%%%%%%%%%%%%%%%
\begin{lemma}
\label{lem3:main}
Let $(A,\mu,\alpha)$ be a multiplicative up to $(n-1)$st Hom-power associative algebra for some $n \geq 5$.  Then
\begin{equation}
\label{ii}
x^{n-4,4} = 11x^{n-2,2} - 10x^n
\end{equation}
for all $x \in A$.
\end{lemma}
%%%%%%%%%%%%%%%%%%

\begin{proof}
A calculation similar to the proof of Lemma \ref{lem1:main}, applied to the special case
\begin{equation}
\label{Fii}
F\left(\alpha^{n-5}(x^2), \alpha^{n-4}(x), \alpha^{n-4}(x), \alpha(x^{n-4})\right) = 0
\end{equation}
of \eqref{F0}, yields
\begin{equation}
\label{ii'}
6x^{n-4,4} + 6x^{n-2,2} + 12x^n
= 16x^{n-3,3} + 8x^{n-2,2}.
\end{equation}
Indeed, the left and the right sides of \eqref{ii'} are, respectively, $F_L$ and $F_R$ applied to the same variables as in \eqref{Fii}.  Continuing \eqref{ii'} we have:
\[
\begin{split}
3x^{n-4,4}
&= 8x^{n-3,3} + x^{n-2,2} - 6x^n\\
&= 8\left(4x^{n-2,2} - 3x^n\right) + x^{n-2,2} - 6x^n \quad\text{(by \eqref{i})}\\
&= 33x^{n-2,2} - 30x^n.
\end{split}
\]
Dividing by three, we obtain the desired equality \eqref{ii}.
\end{proof}

In the next two Lemmas, we assume that $n \geq 6$, which is sufficient in view of Lemma \ref{lem2:main}.  In the following result, part of the $n$th Hom-power associativity of $A$ is established.

%%%%%%%%%%%%%%%%%%
\begin{lemma}
\label{lem4:main}
Let $(A,\mu,\alpha)$ be a multiplicative up to $(n-1)$st Hom-power associative algebra for some $n \geq 6$.  Then
\begin{equation}
\label{v}
x^n = x^{n-2,2}
\end{equation}
for all $x \in A$.
\end{lemma}
%%%%%%%%%%%%%%%%%%

\begin{proof}
Pick $x \in A$.  A calculation similar to the proof of Lemma \ref{lem1:main}, applied to the special case
\[
F\left(\alpha^{n-5}(x^2), \alpha^{n-5}(x^2), \alpha^{n-4}(x), \alpha^2(x^{n-5})\right) = 0
\]
of \eqref{F0}, yields
\begin{equation}
\label{iii}
3x^n = -6x^{n-2,2} + 8x^{n-3,3}+ 4x^{n-4,4} - 3x^{n-5,5}.
\end{equation}
Likewise, the special case
\[
F\left(\alpha^{n-6}(x^3), \alpha^{n-4}(x), \alpha^{n-4}(x), \alpha^2(x^{n-5})\right) = 0
\]
of \eqref{F0} yields
\begin{equation}
\label{iv}
6x^n = 4x^{n-2,2} - 3x^{n-3,3} + 8x^{n-4,4} - 3x^{n-5,5}.
\end{equation}
Subtracting \eqref{iii} from \eqref{iv} and using \eqref{i} and \eqref{ii}, we obtain:
\[
\begin{split}
3x^n &= 10x^{n-2,2} - 11x^{n-3,3} + 4x^{n-4,4}\\
&= 10x^{n-2,2} - 11\left(4x^{n-2,2} - 3x^n\right) + 4\left(11x^{n-2,2} - 10x^n\right)\\
&= 10x^{n-2,2} - 7x^n.
\end{split}
\]
This is equivalent to
\[
10x^n = 10x^{n-2,2},
\]
which yields the desired equality \eqref{v}.
\end{proof}

The following result will be used in the induction step of the proof of Theorem \ref{mainthm}.

%%%%%%%%%%%%%%%%%
\begin{lemma}
\label{lem5:main}
Let $(A,\mu,\alpha)$ be a multiplicative up to $(n-1)$st Hom-power associative algebra for some $n \geq 6$.  Suppose $k$ is an integer in the range $1 \leq k \leq n/2 - 2$.  Then
\begin{equation}
\label{vi}
3x^{n-(k+2),k+2} = -2x^n + 8x^{n-(k+1),k+1} - 3x^{n-k,k}
\end{equation}
for all $x \in A$.
\end{lemma}
%%%%%%%%%%%%%%%%%

\begin{proof}
Pick $x \in A$.  A calculation similar to the proof of Lemma \ref{lem1:main}, applied to the special case
\[
F\left(\alpha^{n-k-3}(x^k), \alpha^{n-4}(x), \alpha^{n-4}(x), \alpha^{k-1}(x^{n-k-2})\right) = 0
\]
of \eqref{F0}, yields
\[
3x^{n-(k+2),k+2} = -6x^n + 8x^{n-(k+1),k+1} - 3x^{n-k,k} + 4x^{n-2,2}.
\]
To finish the proof, one applies Lemma \ref{lem4:main} to replace $x^{n-2,2}$ by $x^n$ in the previous line.
\end{proof}

We can now finish the proof of Theorem \ref{mainthm}

\begin{proof}[Proof of Theorem \ref{mainthm}]
The $n=5$ case of Theorem \ref{mainthm} was established in Lemma \ref{lem2:main}.  So we may assume, without lost of generality, that $n \geq 6$.  Pick $x \in A$.  We must prove the $n$th Hom-power associativity condition
\begin{equation}
\label{npa}
x^n = x^{n-i,i}
\end{equation}
for $1 \leq i \leq n-1$.  In view of Proposition \ref{prop:n-1}, we only need to prove \eqref{npa} when $1 \leq i \leq n/2$.  We prove this by a finite induction.

The $i=1$ case of \eqref{npa} is true by the definition of the Hom-power $x^n$.  The $i=2$ case of \eqref{npa} is true by Lemma \ref{lem4:main}.  For the induction step, suppose $1 \leq k \leq n/2-2$ and that \eqref{npa} is true for $1 \leq i \leq k+1$.  We want to prove \eqref{npa} when $i = k+2$. By Lemma \ref{lem5:main} and the induction hypothesis, we have:
\[
\begin{split}
3x^{n-(k+2),k+2} &= -2x^n + 8x^{n-(k+1),k+1} - 3x^{n-k,k}\\
&= -2x^n + 8x^n - 3x^n\\
&= 3x^n.
\end{split}
\]
Dividing by three in the above calculation yields \eqref{npa} when $i = k+2$, finishing the induction and the proof of Theorem \ref{mainthm}.
\end{proof}

%%==============%%
%%              %%
%%  References  %%
%%              %%
%%==============%%


\begin{thebibliography}{AA}
\bibitem{albert1}
A.A. Albert, On the power-associativity of rings, Summa Brasil. Math. 2 (1948) 21-32.

\bibitem{albert2}
A.A. Albert, Power-associative rings, Trans. Amer. Math. Soc. 64 (1948) 552-593.

\bibitem{arenas}
M. Arenas, An algorithm for associative bilinear forms, Linear Alg. Appl. 430 (2009) 286-295.

\bibitem{baez}
J.C. Baez, The octonions, Bull. Amer. Math. Soc. 39 (2002) 145-205.

\bibitem{carmody}
K. Carmody, Circular and hyperbolic quaternions, octonions, and sedenions, Appl. Math. Comp. 28 (1988) 47-72.

\bibitem{cawagas}
R.E. Cawagas, Loops embedded in generalized Cayley algebras of dimension $2^r,r\geq 2$, Int. J. Math. Math. Sci. 28 (2001) 181-187.

\bibitem{gm}
M. Goze and E. Remm, Lie-admissible algebras and operads, J. Algebra 273 (2004) 129-152.

\bibitem{hls}
J.T. Hartwig, D. Larsson, and S.D. Silvestrov, Deformations of Lie algebras using $\sigma$-derivations, J. Algebra 295 (2006) 314-361.

\bibitem{imaeda}
K. Imaeda and M. Imaeda, Sedenions: algebra and analysis, Appl. Math. Comp. 115 (2000) 77-88.

\bibitem{mak}
A. Makhlouf, Hom-alternative algebras and Hom-Jordan algebras, Int. Elect. J. Alg. 8 (2010) 177-190.

\bibitem{ms}
A. Makhlouf and S. Silvestrov, Hom-algebra structures, J. Gen. Lie Theory Appl. 2 (2008) 51-64.

\bibitem{ms2}
A. Makhlouf and S. Silvestrov, Hom-algebras and Hom-coalgebras, arXiv:0811.0400.

\bibitem{schafer}
R.D. Schafer, An introduction to nonassociative algebras, Dover Pub., New York, 1995.

\bibitem{velasco}
M.V. Velasco, Spectral theory for non-associative complete normed algebras and automatic continuity, J. Math. Anal. Appl. 351 (2009) 97-106.

\bibitem{yau}
D. Yau, Enveloping algebras of Hom-Lie algebras, J. Gen. Lie Theory Appl. 2 (2008) 95-108.

\bibitem{yau2}
D. Yau, Hom-algebras and homology, J. Lie Theory 19 (2009) 409-421.

\bibitem{yau3}
D. Yau, The Hom-Yang-Baxter equation, Hom-Lie algebras, and quasi-triangular bialgebras, J. Phys. A 42 (2009) 165202 (12pp).

\bibitem{yau4}
D. Yau, Hom-bialgebras and comodule Hom-algebras, Int. Elect. J. Alg. 8 (2010) 45-64.

\bibitem{yau12}
D. Yau, Hom-Maltsev, Hom-alternative, and Hom-Jordan algebras, arXiv:1002.3944.

\bibitem{yau5}
D. Yau, Hom-Novikov algebras, arXiv:0909.0726.

\bibitem{yau6}
D. Yau, The Hom-Yang-Baxter equation and Hom-Lie algebras, arXiv:0905.1887.

\bibitem{yau7}
D. Yau, The classical Hom-Yang-Baxter equation and Hom-Lie bialgebras, arXiv:0905.1890.

\bibitem{yau8}
D. Yau, Infinitesimal Hom-bialgebras and Hom-Lie bialgebras, arXiv:1001.5000.

\bibitem{yau9}
D. Yau, Hom-quantum groups I: quasi-triangular Hom-bialgebras, 	 arXiv:0906.4128.

\bibitem{yau10}
D. Yau, Hom-quantum groups II: cobraided Hom-bialgebras and Hom-quantum geometry, arXiv:0907.1880.

\bibitem{yau11}
D. Yau, Hom-quantum groups III: representations and module Hom-algebras, arXiv:0911.5402.

\bibitem{yau13}
D. Yau, On $n$-ary Hom-Nambu and Hom-Nambu-Lie algebras, arXiv:1004.2080.

\bibitem{yau14}
D. Yau, On $n$-ary Hom-Nambu and Hom-Maltsev algebras, arXiv:1004.4795.

\bibitem{yau15}
D. Yau, A Hom-associative analogue of $n$-ary Hom-Nambu algebras, arXiv:1005.2373.
\end{thebibliography}
\end{document}